\documentclass{article}

\usepackage{amsmath}
\usepackage{amsthm}
\usepackage{amssymb}
\usepackage{bbm}
\usepackage{verbatim}
\usepackage{graphicx}
\usepackage{afterpage}
\usepackage{etoolbox}
\usepackage{float}
\usepackage{algorithmic}
\usepackage[algosection,algoruled,vlined]{algorithm2e}
\usepackage[noadjust]{cite}
\usepackage{url}
\usepackage{authblk}
\usepackage{hyperref}
\usepackage{breqn}
\usepackage[utf8]{inputenc}
\usepackage{color}
\usepackage{caption}
\usepackage{subcaption}
\usepackage{tikz}
	\usetikzlibrary{shapes,arrows}
	\usetikzlibrary{positioning}
	\usetikzlibrary{calc,shapes, positioning}
	\usetikzlibrary{patterns}

\numberwithin{equation}{section}
\numberwithin{figure}{section}
\numberwithin{table}{section}

\makeatletter
\renewcommand{\p@subfigure}{\thefigure}
\makeatother

\newtheorem{theorem}{Theorem}[section]

\newtheorem{lemma}[theorem]{Lemma}
\newtheorem{corollary}[theorem]{Corollary}
\newtheorem{rem}[theorem]{Remark}
\newtheorem{definition}{Definition}[section]

\theoremstyle{remark}
\newtheorem*{remark}{Remark}

\usepackage[a4paper, margin=1.1in]{geometry}
\newcommand{\nnz}{\operatorname{nnz}}
\newcommand{\newdef}{\stackrel{\Delta}{=}}
\newcommand{\norm}[1]{\left \| #1 \right \|}
\newcommand{\pr}[2]{\langle {#1} , {#2} \rangle}

\def \e {\varepsilon}

\def \l {\lambda}

\def \R {\mathbb{R}}
\def \N {\mathbb{N}}

\def \E {\mathbb{E}}
\def \P {\mathbb{P}}
\def \MM {\mathcal{M}}
\def \NN {\mathcal{N}}

\begin{document}

\title{Matrix Decompositions using sub-Gaussian Random Matrices}

\author
{Yariv Aizenbud${^1}$~~Amir Averbuch${^2}$\\
${^1}$School of Applied Mathematics, Tel Aviv University, Israel\\
${^2}$School of Computer Science, Tel Aviv University, Israel
}
\maketitle

\begin{abstract}

In recent years, several algorithms, which approximate matrix
decomposition, have been developed. These algorithms are based on
metric conservation features for linear
spaces of random projection types. We show that an i.i.d sub-Gaussian matrix with large probability to have zero entries is metric conserving. We also present a new algorithm, which achieves with high probability, a rank $r$ decomposition approximation for an $m \times n$ matrix that has an asymptotic complexity like state-of-the-art algorithms.
We derive an error bound that does not depend on the first $r$ singular values.
Although the proven error bound is not as tight as the state-of-the-art bound, experiments show that the proposed algorithm is faster in practice, while getting the same error rates as the state-of-the-art algorithms get.

    \smallskip

\end{abstract}

\section{Introduction}

Dimensionality reduction by randomized linear maps preserves metric
features. The Johnson-Lindenstrauss Lemma (JL)
\cite{johnson1984extensions} shows that there is a random
distribution of linear dimensionality reduction operators that
preserves, with bounded error and high probability, the norm of a
set of vectors. For example, Gaussian random
matrices satisfy this property.

JL Lemma was extended in the following way. While
the classical formulation dealt  with norm conservation of sets of
vectors, the JL-based extension deals with a subspace of a vector space. This
extension is considered for example in \cite{tropp2011improved},
where it shows that Fourier based random matrices of size $n \times
\mathcal{O}(r\log r)$ conserves the norm of all the vectors from a
vector space of dimension $r$. Similar results for sparse matrices
distribution are given in
\cite{nelson2013osnap,dasgupta2010sparse,kane2014sparser,clarkson2013low}.

In recent years,  several algorithms that approximate matrix
decomposition, which are based on norm conservation, have been
developed. The idea is roughly as follows: A randomly drawn matrix $\Omega$, which projects the original matrix into a lower dimension, is used. The decomposition is calculated in the low dimensional space. Then, this decomposition is mapped into the matrix original size. It is shown in \cite{martinsson2011randomized,shabat2013randomized}  how to
use random Gaussian matrices in order to find, with high
probability, an approximated interpolative decomposition,
singular value decomposition (SVD) and LU decomposition. FFT-based random matrices, which
approximate matrix decompositions, are described in
\cite{woolfe2008fast}.  The special structure of the FFT-based
distribution  provides a  fast matrix multiplication that yields a
faster algorithm than the algorithms in
\cite{martinsson2011randomized}. A comprehensive review of these
ideas (and many more) is given in \cite{halko2011finding}. The
algorithm in \cite{clarkson2013low} uses a sparse random matrix
distribution that makes the matrix multiplication step in the
algorithm even faster than what the FFT-based matrices provide.

In this paper, we show that the class of matrices with i.i.d sub-Gaussian entries satisfy
the image conservation property even when the probability for a zero entry grows with the size of the matrix. Additionally, we construct
fast SVD and LU decomposition algorithms with bounded error and asymptotic complexity equal to the asymptotic complexity of the state-of-the-art algorithm. Although the asymptotic complexity is the same, the practical running time of the presented algorithms is lower than the existing algorithms. Since the random projections are matrices with i.i.d entries, it is not required to set the dimension $k$ of the projection in advance. It is possible, although not elaborated in this paper, to increase $k$ iteratively, until the resulting approximation is in the required accuracy. Stronger bounds for the case of sparse-Bernoulli random matrices are shown in \cite{cohen2016nearly}\footnote{ The results on sub-Gaussian random matrices in this paper were derived couple of months before the paper of \cite{cohen2016nearly} was brought to our attention.}

We denote by $M_{n \times m}$ the set of $n$ by $m$ matrices. 
We call a  rectangular random matrix distribution $\MM$ an
{\it metric conserving} distribution if for any $A \in M_{n\times m}$
a randomly chosen $\Omega \in M_{m \times k}$ from $\MM$,
the image of $A\Omega$ is similar to the image of $A$. Three main
parameters related to this property are the dimension $k$ of $\Omega$ (the smaller the
better), the ``distance" between the images of $A\Omega$ and $A$ and the
probability for which the image conservation is valid. It is obvious
that these parameters are connected. Distributions, which conserve the norm allowing an error $(1+\varepsilon)$ of the theoretical bound, are called {\it oblivious subspace embedding (OSE)} (\cite{nelson2013osnap}).

The theoretical bound for a rank $r$ approximation of a matrix A in $L_2$ norm is $\sigma_{r+1}(A)$ and in Frobenius norm it is $\Delta_{r+1}$, where $\sigma_r(A)$ is the $r$th largest singular value of $A$ and $\Delta_r(A) \stackrel{\Delta}{=} (\sum_{l=r}^{n}\sigma_l^2(A))^{1/2}$. Three important results related to the above parameters, which deal with  metric conserving distributions in the context of randomized decomposition algorithms, are: 
1. Achieving an accuracy of $\mathcal{O}_\sigma(\sigma_{r+1}(A))$ for a rank $r$ measured in $L_2$ norm
with high probability, is described in \cite{halko2011finding,martinsson2011randomized}. To
achieve this accuracy with high probability, the required $\Omega$ can be an i.i.d Gaussian matrix of size $\mathcal{O}(r)$. 
2. Achieving an accuracy of $\mathcal{O}_\sigma(\sigma_{r+1}(A))$ for a rank $r$ measured in $L_2$ norm
with high probability, is described in \cite{halko2011finding,woolfe2008fast}. To
achieve this accuracy with high probability, $\Omega$ can be an FFT-based matrix of size $\mathcal{O}(r \log r)$. 
3. The result in \cite{nelson2013osnap} achieves accuracy of $(1+\e)\Delta_{r+1}(A) $ with high probability measured in Frobenius norm. While $\Omega$ is drawn from a sparse distribution, its size is assumed to be not less than $\mathcal{O}(r^2/\e^2)$. In fact, for sparse matrices distribution, a lower bound for the size of $\Omega$ is provided in \cite{nelson2013sparsity}. 

We show in Section \ref{sec:rectangular} that for the class of matrices with i.i.d 
sub-Gaussian entries, the size of $\Omega$, which is needed to achieve an
accuracy $\mathcal{O}_\sigma(\sigma_{r+1}(A))$  measured in $L_2$ norm. We also show its dependency on the probability to have a zero entry. By choosing a sparse matrix distribution to
be sub-Gaussian, we were able to perform a fast matrix
multiplication while having a small size $\Omega$. It is shown in 
\cite{dirksen2014dimensionality} that this class of sub-Gaussian matrices of size $\mathcal{O}(r/\varepsilon^2)$ with constant probability distribution is an OSE. In this paper, we provide a bound for the case where the distribution depends on the size of the matrix.

The state-of-the-art result for rank $r$ approximation algorithm appears in \cite{clarkson2013low}. It describes how to use a sparse embedding matrix to construct an algorithm that finds for any matrix $A \in M_{m \times n}$ and any rank $r$, with high probability, an SVD approximation of rank $r$. Namely, orthogonal  $U,V^*$ and a diagonal matrix $\Sigma$ are formed such that  $\|A-U\Sigma V^*\|_F \leq (1+\e)\Delta_{r+1}(A) $. 
Although the algorithm in \cite{halko2011finding} uses a smaller $\Omega$ than \cite{clarkson2013low}, the algorithm in \cite{clarkson2013low} is asymptoticly faster than the algorithm in \cite{halko2011finding} because of the sparse nature of the projection.

We describe in Section
\ref{sec:Sparse_Randomized_SVD} an algorithm that for each $A\in
M_{m\times n}$ outputs with high probability a low rank SVD approximation that is built from
$U, \Sigma$ and $V$. The algorithm works with any metric conserving or OSE random distribution.
The size $k$ of the random embedding in the algorithm depends on the probability $p$ for having a zero entry. The complexity of the algorithm when using i.i.d sub-Gaussian random matrix projections is $\mathcal{O}(\nnz(A)pk + (m+n)k^2)$ where $\nnz (A)$ denotes the {\it number of non-zeros} in $A$ and $ k = \mathcal{O}(\frac{1}{p^3}\ln r)$. For sparse embedding matrix distribution as in \cite{nelson2013osnap}, the complexity of the algorithm in Section \ref{sec:Sparse_Randomized_SVD} is the same as in \cite{nelson2013osnap}.
This algorithm guarantees with high probability that
$\Vert A-U\Sigma V^* \Vert_2 \le \mathcal{O}_\sigma(\sigma_{r+1}(A))$. Although the guaranteed error bound is less tight than the one in \cite{clarkson2013low}, we show in Section \ref{subsec:numerical_results} that in practice our algorithm reaches the same error in less time.

The randomized LU decomposition algorithm in \cite{sparseLU} is based on the ideas from \cite{clarkson2013low}. 
We show in Section \ref{sec:Sparse_Randomized_LU} that it is also valid when random matrices from a sub-Gaussian
distribution are chosen with the complexity and error bound equal to 	
those from the SVD decomposition.

The paper has the following structure: In Section 2, we present the necessary mathematical preliminaries. In Section 3, we show that i.i.d sub-Gaussian random matrices are metric conserving and in Section 4 we describe the SVD algorithm and show that the LU algorithm in \cite{sparseLU} is valid with i.i.d sub-Gaussian random matrices. In section 5, we present the numerical results of the described SVD algorithm.

\section{Preliminaries}

\subsection{The $\e$-Net}\label{sec:epsNet}

$\e$-net is defined in Definition \ref{def:eps_net}. Its size is  bounded by Lemma
\ref{lem:eps_net_size} that is proved in
\cite{rudelson2009smallest}.
Throughout the paper, $ S^{n-1}$ denotes the $(n-1)$-sphere in $\R^{n}$.

\begin{definition}\label{def:eps_net}
    Let $(T,d)$ be a metric space and let $K \subset T$. A set
    $\NN \subset T$ is called $\e$-net of $K$ if for all $x \in K$ there exists $y \in
    \NN$
    such that $ \ d(x,y)<\e$.

\end{definition}
\begin{lemma}[Proposition 2.1 in \cite{rudelson2009smallest}] \label{lem:eps_net_size}
    For any $\e<1$, there exists an $\e$-net $\NN$  of $S^{n-1}$ such that
    $$
    |\NN| \le 2n \left(1+ \frac{2}{\e} \right)^{n-1}.
    $$
\end{lemma}
\begin{remark}
    It follows that for sufficiently large $n$, the size of $1/2$ - net of $S^{n-1}$ has at most
    $$
    2n \left(1+ \frac{2}{1/2} \right)^{n-1} = 2n \cdot 5^{n-1} \leq 6^n
    $$
    points.
\end{remark}

\subsection{Compressible and Incompressible Vectors} \label{sec:comp_incomp}
\begin{definition}
	A vector $v \in \R^n$ is called \it{$(\eta,\e)$-incompressible} if $ \sum\limits_{j:|v_j| \leq \e} |v_j|^2 \geq \eta^2$ and compressible otherwise.
\end{definition}

\begin{lemma}\label{lem:eps_net_compressible_size}
	Let $U \subset \R^n$ a subspace of dimension $r$. Let $\NN $ be an $\e_{net}$-net  of the set of  $(\eta, \e_c)$-compressible vectors in $U$. Then, 
	$$
	|\NN| \leq r^{\frac{1}{\e_c^2}} \eta^{r-\frac{1}{\e_c^2}} \left(\frac{c_{net}}{\e_{net}}\right)^r
	$$
	for an absolute constant $c_{net}$.
\end{lemma}
\begin{proof}
	The $(\eta, \e_c)$-compressible vectors are in an $\eta$ distance from a sparse vector with no more than $\frac{1}{\e_c^2}$ non-zero coordinates. For small enough $\eta$, the volume of $\eta$-balls around \mbox{$\frac{1}{\e_c^2}$ - sparse} vectors is $r^{\frac{1}{\e_c^2}} \eta^{r-\frac{1}{\e_c^2}} $. The same arguments from the proof of Lemma \ref{lem:eps_net_size} show that the number of points in an $\e_{net}$-net of this volume is not more than $r^{\frac{1}{\e_c^2}} \eta^{r-\frac{1}{\e_c^2}} \left(\frac{c_{net}}{\e_{net}}\right)^r$.
\end{proof}

\subsection{Sub-Gaussian Random Variables} \label{sec: subgaussian}

In this section, we introduce the sub-Gaussian random variables with
some of their properties. Sub-Gaussian variables are an important
class of random variables that have strong tail decay properties.
This class contains, for example, all the bounded random variables
and the normal variables.

\begin{definition}
    A random variable  $X$ is called sub-Gaussian  if
    there exists constants $v$ and $C$  such that for any $t>0$, 
    $\P(|X|>t) \le C e^{-vt^2}$ and $X$ has a non-zero variance.
        A random variable $X$ is called centered if $\E X=0$.
\end{definition}

\begin{remark}
	For convenience, we use the term \it{sub-Gaussian matrix} for a matrix with i.i.d sub-Gaussian entries.
\end{remark}
Many non-asymptotic results on a sub-Gaussian matrix distribution have recently
appeared.
A survey of this topic appears in
\cite{rudelson2014recent,vershynin2010introduction}. 

The following facts, proved in
\cite{rudelson2009smallest,rudelson2014recent,vershynin2010introduction,rudelson2008littlewood,litvak2005smallest},
are used in the paper:
\begin{enumerate}
    \item Linear combination of centered sub-Gaussian variables is also sub-Gaussian. This is stated in Theorem \ref{thm:Hoeffding}. The inequality in this theorem is similar to
     Hoeffding inequality \cite{hoeffding1963probability}.
    \item The bound for the first singular value of a sub-Gaussian random matrix is given in Theorem \ref{thm:norm}.
    \item The probability bound for the sum of centered sub-Gaussian variables to be small  is given in Theorem \ref{lem:LPRT}.
\end{enumerate}
Formally,

\begin{theorem}
    \label{thm:Hoeffding}
    Let $X_1 , \ldots, X_n$ be independent centered sub-Gaussian  random variables.
    Then, for any $a_1, \ldots, a_n \in \R$
    $$
    \P \left (\left | \sum_{j=1}^n a_j X_j \right | >t \right )
    \le  2 \exp \left ( - \frac{c t^2}{ \sum_{j=1}^n a_j^2} \right ).
    $$
\end{theorem}

\begin{theorem}
    \label{thm:norm}
    Let $\Omega$ be a $k \times n$,~$n \ge k$, random matrix whose entries
    are i.i.d  centered sub-Gaussian random variable. Then,  $
    \P \big( \sigma_1(\Omega) > t \sqrt{n} \big)    \le e^{-c_0 t^2
    n}$ holds for  $t \geq C_0$.
\end{theorem}

Since we are interested in sparse matrices, the following definition is useful.
\begin{definition} \label{def:sub_gauss_X}
	A sub-Gaussian random variables $X$ is represented by a combination of a centered sub-Gaussian random variable $\frac{1}{\sqrt{p}}Z$ with $\P(Z=0)=0,~ \E(Z^2) = 1$ with probability $p$ and $0$ otherwise. Note that $\E(X) = 0, \E(X^2) = 1, \E(X^3) = \frac{\E(Z^3)}{\sqrt{p}}$ and $\E(X^4) = \frac{\E(Z^4)}{p}$.
\end{definition} 

\begin{lemma}
	\label{lem:sum_sub_gaussian_moments}
	Let $X_1 , \ldots, X_n$ be independent centered sub-Gaussian random variables defined as a combination of a centered sub-Gaussian $\frac{1}{\sqrt{p}}Z$ with $\P(Z=0)=0$ and $E(Z^2) = 1$ with probability $p$ and $0$ otherwise.
	Then, for any $(a_1 \ldots a_n) \in S^{n-1}$ the third and forth moment (skewness and kortosis) of $(\sum\limits_{i=1}^{n} a_i X_i)$ are bounded by
	$$
	\E\left((\sum\limits_{i=1}^{n} a_i X_i)^3\right) \leq \frac{\E(Z^3)}{\sqrt{p}}
	$$
	and
	$$
	\E\left((\sum\limits_{i=1}^{n} a_i X_i)^4\right) \leq \frac{\E(Z^4)+1}{p} \newdef \frac{z_4}{p}.
	$$
\end{lemma}
\begin{proof}
	$$
		\E\left((\sum\limits_{i=1}^{n} a_i X_i)^3\right) = \sum\limits_{i=1}^{n} a_i^3 \E(X_i^3) + \sum\limits_{i,j=1, i \neq j}^{n} a_i^2 a_j \E(X_i^2) \E(X_j) \leq \E(X^3) = \frac{\E(Z^3)}{\sqrt{p}}.
	$$
	$$
		\E\left((\sum\limits_{i=1}^{n} a_i X_i)^4\right) = \sum\limits_{i=1}^{n} a_i^4 \E(X_i^4) + \sum\limits_{i,j=1, i \neq j}^{n} a_i^2 a_j^2 \leq \E(X^4)+1 = \frac{\E(Z^4)+p}{p}.
	$$
	Since $p \leq 1$, the proof is completed.
\end{proof}

\begin{lemma} 
	\label{lem:LPRT}
	Let $X_1 , \ldots, X_n$ be an i.i.d
	centered sub-Gaussian random variable as in Definition \ref{def:sub_gauss_X}. For every coefficients vector (in particular for a compressible vector) $a= (a_1,\ldots,a_n) \in
	S^{n-1}$,
	the random sum $S = \sum_{i=1}^n a_i X_i$ satisfies
	$\P(|S| < \lambda) \le 1- p \frac{(1-\l^2)^2}{z_4} $.
\end{lemma}

\begin{proof}
	Let $0< \l < (\E S^2)^{1/2}=1$. By the Cauchy--Schwarz inequality,
	$$
	\E S^2 = \E S^2 \mathbf{1}_{[-\l,\l]}(S)+ \E S^2 \mathbf{1}_{\R \setminus [-\l,\l]}(S)
	\le \l^2+ \left(\E S^4 \right)^{1/2} \P (|S| > \l)^{1/2}.
	$$
	This leads to the  Paley--Zygmund inequality:
	$$
	\P (|S| > \l) \ge \frac{(\E S^2 - \l^2)^2}{\E S^4}
	= \frac{(1-\l^2)^2}{\E S^4}.
	$$
	By Theorem \ref{thm:Hoeffding}, the random variable $S$ is sub-Gaussian.
	By Lemma \ref{lem:sum_sub_gaussian_moments}, $ \E S^4 \leq \frac{z_4}{p}$ where $z_4 = \E Z^4 + 1$.
	To complete the proof 	
	$$
	\P (|S| < \l) \le 1- \frac{(1-\l^2)^2}{\E S^4} = 1-p\frac{(1-\l^2)^2}{z_4}.
	$$
	In particular, for $\l = 1/2$ we have 	$\P (|S| <  1/2) \le 1-z_4'p$ for $z_4' = \frac{9}{16z_4}$

\end{proof}

\begin{lemma}\label{lem:k_choose_alpha_k}
	For any $0<\alpha<1$, there is $c_s$ such that for any k, ${k \choose \alpha k} < c_s^{\alpha k \ln (\frac{1}{\alpha}-1)} < c_s^{\alpha k \ln \frac{1}{\alpha}}$.
\end{lemma}
\begin{proof}
	We use the Stirling formula to estimate $\ln  {k \choose \alpha k}$.
	\begin{align*}
	\ln  {k \choose \alpha k} & =  k \ln k - k - (\alpha k \ln (\alpha k) - \alpha k) - ((k- \alpha k) \ln (k- \alpha k) - (k- \alpha k) ) + \mathcal{O}(\ln k)\\
	& = k \ln k - \alpha k \ln (\alpha k) - k \ln(k- \alpha k) + \alpha k \ln(k-\alpha k) + \mathcal{O}(\ln k) \\
	& = k \ln k - \alpha k \ln \alpha k - k \ln k(1-\alpha) + \alpha k \ln (\alpha k (\frac{1}{\alpha} - 1)) + \mathcal{O}(\ln k)\\
	& = \alpha k \ln (\frac{1}{\alpha} - 1) - k \ln (1- \alpha) + \mathcal{O}(\ln k)\\ 
	& \sim \alpha k \ln   (\frac{1}{\alpha} - 1) - k(-\alpha - \frac{\alpha^2}{2} - \cdots)  \sim  \alpha k \ln (\frac{1}{\alpha}-1).
	\end{align*}
	
\end{proof}

Lemma \ref{lem:berry_essen_var} follows from Berry-Essen's theorem \cite{berry1941accuracy,esseen1942liapounoff} in a similar fashion to the derivations in \cite{tao2012topics}.
\begin{lemma} \label{lem:berry_essen_var}
	For $S = \sum_{i=1}^n a_i X_i$ where $X_i$ are i.i.d random variables with $E(X) = 0, E(X^2) = 1$ and $\sum\limits_{i=1}^n a_i^2 = 1$, then for all $r$
	\begin{equation}
	\begin{split}
	\sup\limits_t \P\left(|\sum\limits_{i=1}^n a_i X_i-t|  < r \right)  
	\leq
	\frac{r}{\sqrt{\sum\limits_{i=1}^n a_i^2}} + 2C_{BE} \frac{\E(X^3) \sum_{i=1}^n |a_i|^3}{\left(\sum\limits_{i=1}^n a_i^2 \right)^{3/2}}
	\end{split}
	\end{equation}
	holds.
\end{lemma}
\begin{proof}
	Let $N$ be a standard normal variable. From Barry-Essen's theorem follows that for all $r$
	$$
	\left| \P\left(\frac{\sum_{i=1}^n a_i X_i}{\sqrt{\sum_{i=1}^n a_i^2}} < \frac{r}{\sqrt{\sum_{i=1}^n a_i^2}}\right) - \P\left( N<\frac{r}{\sqrt{\sum_{i=1}^n a_i^2}}\right) \right| \leq C_{BE} \frac{\E(X^3) \sum_{i=1}^n |a_i|^3}{\left(\sum_{i=1}^n a_i^2 \right)^{2/3}}.
	$$
	Thus, for any $t$,
	
	\begin{multline} \label{eq:barry_essen_difference}
		\left| \P\left(\frac{|\sum_{i=1}^n a_i X_i-t|}{\sqrt{\sum_{i=1}^n a_i^2}}  < \frac{r}{\sqrt{\sum_{i=1}^n a_i^2}}\right)  - \P\left( |N-t|<\frac{r}{\sqrt{\sum_{i=1}^n a_i^2}}\right) \right| 
		\\
		\begin{aligned}
			&=
			\left| 
			\P\left(\frac{\sum_{i=1}^n a_i X_i}{\sqrt{\sum_{i=1}^n a_i^2}} < \frac{t+r}{\sqrt{\sum_{i=1}^n a_i^2}}\right) - 
			\P\left( N<\frac{t+r}{\sqrt{\sum_{i=1}^n a_i^2}}\right)  - \right. \\  
			&\qquad\qquad \qquad\qquad
			\left. \P\left(\frac{\sum_{i=1}^n a_i X_i}{\sqrt{\sum_{i=1}^n a_i^2}} < \frac{t-r}{\sqrt{\sum_{i=1}^n a_i^2}}\right) - 
			P\left( N<\frac{t-r}{\sqrt{\sum_{i=1}^n a_i^2}}\right) 
			\right| 
			\\
			&\leq 2C_{BE} \frac{\E(X^3) \sum_{i=1}^n |a_i|^3}{\left(\sum_{i=1}^n a_i^2 \right)^{3/2}}.
		\end{aligned}
	\end{multline}
	
	By rewriting \ref{eq:barry_essen_difference}, we have
	\begin{equation}
	\begin{split}
	\P\left(|\sum_{i=1}^n a_i X_i-t|  < r \right)  
	\leq
	\P\left( |N-t|<\frac{r}{\sqrt{\sum_{i=1}^n a_i^2}}\right) + 2C_{BE} \frac{\E(X^3) \sum_{i=1}^n |a_i|^3}{\left(\sum_{i=1}^n a_i^2 \right)^{3/2}}.
	\end{split}
	\end{equation}
	For any $t$, $\P\left( |N-t|<\frac{r}{\sqrt{\sum_{i=1}^n a_i^2}}\right) \leq \P\left( |N|<\frac{r}{\sqrt{\sum_{i=1}^n a_i^2}}\right) < \frac{r}{\sqrt{\sum_{i=1}^n a_i^2}} $.
\end{proof}

\section{Metric conservation of sub-Gaussian random matrices }\label{sec:rectangular}

The main goal of this section is to show that for any matrix $A$ and for a sub-Gaussian matrix $\Omega$, the image of $A\Omega$ is
``close'' to the image of $A$ with high probability, or, in other
words, $\Omega$ preserves the  geometry. Namely, if $Q$ is an
orthogonal basis for $A\Omega$, then $\|A-QQ^*A\|_2$ is
small.

In order to show that the application of a random sub-Gaussian
matrix  preserves the geometry of $A$, we have to bound its behavior
in any subspace of a given dimension $r$. We show in Theorem \ref{thm:max_on_subspase} that
the norm of a random sub-Gaussian matrix in a subspace of dimension $r$ is
bounded from above with high probability. In Lemmas \ref{lem:max_val_one_vec_incomp} and \ref{lem:max_val_one_vec} it is shown that $\Omega$ conserves compressible and incompressible vectors, respectively, from a subspace of dimension $r$. In Theorem
\ref{thm:rectangular_small_val_subspace}, these results are joined to show that the minimal
singular value is bounded from below with high probability. The flow of the proof is based on ideas from the proof of bounds on singular values of Bernoulli random matrix in \cite{tao2012topics} and ideas from \cite{rudelson2014recent}. In Theorem \ref{thm:A_QQA}, these results and the fact that the norm of
a random matrix is also bounded (Theorem \ref{thm:norm}) are used to
show that a sub-Gaussian matrix preserves the geometry.

These are the dependencies among the different theorems in this section:
\tikzstyle{line} = [draw, -latex']
\tikzstyle{block} = [rectangle,text width=1cm, text centered]
\begin{center}
\begin{tikzpicture}
\node [block] (T41) {T \ref{thm:A_QQA}};

\node [block, above right =0.2cm and 0.8cm of T41] (T24) {T \ref{thm:norm}};
\node [block, right =0.8cm of T41] (T36) {T \ref{thm:rectangular_small_val_subspace}};

\node [block, right =0.8cm of T36] (T33) {L \ref{lem:max_val_one_vec}};
\node [block, above =0.2cm of T33] (T34) {T \ref{thm:max_on_subspase}};
\node [block, below =0.2cm of T33] (T32) {L \ref{lem:max_val_one_vec_incomp}};

\node [block, right =0.8cm of T32] (L31) {L \ref{lem:one_row_incompressible_bound}};

\path [line] (T36) -- (T41);
\path [line] (T24) -- (T41);

\path [line] (T33) -- (T36);
\path [line] (T34) -- (T36);
\path [line] (T32) -- (T36);

\path [line] (L31) -- (T32);

\end{tikzpicture}	
\end{center}

\begin{lemma} \label{lem:one_row_incompressible_bound}
	Let $X_1, \ldots, X_n$ be i.i.d centered sub-Gaussian random variables as in Definition \ref{def:sub_gauss_X}. Denote $\e_c = \e_0 \eta \sqrt{p}$. For any $(\eta,\e_c)$-incompressible $a= (a_1,\ldots,a_n) \in S^{n-1}$, the random sum $\sum_{i=1}^n a_i X_i$ satisfies 
	$$
	P(|\sum_{i=1}^n a_i X_i| \leq2 C \e_0 \eta) \leq 2C_1(C) \e_0 \E(Z^3) 
	$$
	where $C$ is a constant that will be chosen later, and $C_1(C)$ depends only on $C$.
\end{lemma}

\begin{proof}
	We recall that $a$ is incompressible if $\sum\limits_{j:|a_j| \leq \e_c} |a_j|^2 \geq \eta^2$. By using Lemma \ref{lem:berry_essen_var} we get:
		$$
		\sup\limits_t \P(|\sum_{i=1}^n a_i X_i-t| \leq r ) \leq\frac{r}{\sqrt{\sum\limits_{i=1}^n a_i^2}} + 2C_{BE} \frac{\E(X^3) \sum_{i=1}^n |a_i|^3}{\left(\sum\limits_{i=1}^n a_i^2 \right)^{3/2}}.
		$$
		Note that we can condition out variables, 
		$$
		\sup\limits_t \P(|\sum_{i=1}^n a_i X_i-t| \leq r ) \leq \sup\limits_t \P(|\sum_{i=1}^n a_i X_i-t| \leq r | X_1 = x_1 ).
		$$
		If we condition out all the $X_i$ for which $a_i>\e_c$, we get
		$$
		\sup\limits_t \P(|\sum_{i=1}^n a_i X_i-t| \leq r ) 
		\leq 
		\frac{r}{\sqrt{\sum\limits_{j:|a_j| \leq \e_c} a_j^2}} + 2C_{BE} \frac{\E(X^3) \sum_{j:|a_j| \leq \e} |a_j|^3}{\left(\sum\limits_{j:|a_j| \leq \e_c} a_j^2 \right)^{3/2}} 
		\leq
		\frac{r}{\eta} + 2C_{BE} \frac{\E(X^3) \e_c }{\eta}. 
		$$
		By substituting $r = 2C\e_0 \eta$ we have
		$$
		\sup\limits_t \P(|\sum_{i=1}^n a_i X_i-t| \leq 2C\e_0 \eta ) 
		\leq
		\frac{2C\e_0 \eta}{\eta} + 2C_{BE} \frac{\E(X^3) \e_c }{\eta}. 
		$$
		By using Lemma \ref{lem:sum_sub_gaussian_moments} and by substituting $\e_c = \e_0 \eta \sqrt{p}$ the proof is completed.
\end{proof}

\begin{lemma}[ $\Omega$ conserves incompressible vectors in a subspace] \label{lem:max_val_one_vec_incomp}
	Let $\Omega$ be a $k \times n$ ($n \ge k$) random matrix whose entries are i.i.d centered sub-Gaussian random variable as in Definition \ref{def:sub_gauss_X}. Denote $\e_c = \e_0 \eta \sqrt{p}$.
	Then, for any $(\e_c, \eta)$-incompressible $x \in S^{n-1}$, 
	$$
	\P ( \norm{\Omega x}_2  < 2C \alpha \e_0 \eta \sqrt{k}) \le (C_1(C) \e_0 \E(Z^3))^{k/4}
	$$
	for a constant $\alpha$.
\end{lemma}
\begin{proof}
	The coordinates of the vector $\Omega x$ are independent linear
	combinations of i.i.d. sub Gaussian random variables with incompressible 
	coefficients $(x_1 \ldots x_n) \in S^{n-1}$. Hence, by Lemma \ref{lem:one_row_incompressible_bound}, $\P(|(\Omega x)_j|< 2C \e_0 \eta) \le C_1(C)\e_0 \E(Z^3) =\mu$ for all $j =1 , \ldots, N$.
	
	Assume that $\norm{\Omega x}_2< 2C \e_0 \eta \alpha \sqrt{k}$. Then, $|(\Omega x)_j|< 2C\e_0\eta$ for
	at least $\lfloor(1 - \alpha^2)k \rfloor $ coordinates. 
	Thus ,
	$$
	\begin{array}{ll}
	\P(\norm{\Omega x}_2< 2C \e_0 \eta \alpha  \sqrt{k}) &\leq \P\left(\mbox{at least } \lfloor (1-\alpha^2)k \rfloor \mbox{ coordinates satisfay } |(\Omega x)_j|< 2C \e_0 \eta \right) \\
	& = \sum\limits_{l<\lfloor(1-\alpha^2)k\rfloor}^k \binom{k}{l} \P(|(\Omega x)_1|< 2C \e_0 \eta)^l \left(1-\P(|(\Omega x)_1|< 2C \e_0 \eta)\right)^{k-l}\\
	& \leq \sum\limits_{l<\lfloor(1-\alpha^2)k\rfloor}^k \binom{k}{l} \mu^l \\
	& \leq \mu^{\lfloor(1-\alpha^2)k\rfloor}  \sum\limits_{l<\lfloor(1-\alpha^2)k\rfloor}^k \binom{k}{l}.
	\end{array}
	$$
	If $\alpha$ is sufficiently small, then 
	$\lfloor(1-\alpha^2)k\rfloor>k/2$ and 
	$$
	\P(\norm{\Omega x}_2< 2C \e_0 \eta \alpha  \sqrt{k}) \leq \mu^{k/2} \alpha^2 k \binom{k}{\lfloor(1-\alpha^2)k\rfloor}.
	$$
	For $\alpha$ sufficiently small, $ \alpha^2 k \binom{k}{\lfloor(1-\alpha^2)k\rfloor} \leq \mu^{-k/4}$.
	Thus, $	\P( \norm{\Omega x}_2  < 2C \e_0 \eta \alpha  \sqrt{k})
	\le  \mu^{-k/4} \cdot \mu^{k/2} \le \mu^{k/4}.
	$
\end{proof}

\begin{lemma} [$\Omega$ conserves any vector in a subspace]\label{lem:max_val_one_vec}
	Let $\Omega$ be a $k \times n$ ($n \ge k$) random matrix whose entries are i.i.d centered sub-Gaussian  random variable with variance $1$  as in Definition \ref{def:sub_gauss_X}.
	Let $C$ be a constant that will be chosen later, and let $\eta$ be small enough, such that $\eta^2 \ln \frac{1}{\eta} < c_4 p$. Then, for every $x \in S^{n-1}$, $\P ( \norm{\Omega x}_2  < 2 C \eta \sqrt{k}) \le (1-z_4'p)^{k/4}$ holds for a constant $c_4$.
\end{lemma}
\begin{proof}
	The coordinates of the vector $\Omega x$ are independent linear
	combinations of i.i.d. sub Gaussian random variables with
	coefficients $(x_1 \ldots x_n) \in S^{n-1}$. Hence, for $\l = 1/2$,by Lemma \ref{lem:LPRT}  $\P(|(\Omega x)_j|< 1/2) \le 1-z_4'p$, $j =1 , \ldots, N$.
	
	Assume that $\norm{\Omega x}_2< 2 C \eta \sqrt{k}$. Then, $|(\Omega x)_j|< 1/2$ for
	at least $\lfloor(1-4 \cdot 2^2 C^2  \eta^2)k \rfloor $ coordinates. 
	Thus 
	$$
	\begin{array}{ll}
	\P(\norm{\Omega x}_2< 2 C \eta \sqrt{k}) &\leq \P(\# \lfloor (1-16 C^2  \eta^2)k \rfloor \mbox{ coordinates satisfay } |(\Omega x)_j|< 1/2)  \\
	& = \sum\limits_{l<\lfloor(1-16 C^2 \eta^2)k\rfloor}^k \binom{k}{l} \P(|(\Omega x)_1|<1/2)^l \left(1-\P(|(\Omega x)_1|<1/2)\right)^{k-l}\\
	& \leq \sum\limits_{l<\lfloor(1-16 C^2 \eta^2)k\rfloor}^k \binom{k}{l} (1-z_4'p)^l \\
	& \leq (1-z_4'p)^{\lfloor(1-16 C^2 \eta^2)k\rfloor}  \sum\limits_{l<\lfloor(1-16 C^2 \eta^2)k\rfloor}^k \binom{k}{l}.
	\end{array}
	$$
	If $\eta$ is sufficiently small, then 
	$\lfloor(1-16 C^2 \eta^2)k\rfloor>k/2$ and 
	$$
	\P(\norm{\Omega x}_2 < 2 C \eta \sqrt{k}) \leq (1-z_4'p)^{k/2} 16 C^2 k \binom{k}{\lfloor(1-16 C^2 \eta^2)k\rfloor}.
	$$
	
	From Lemma \ref{lem:k_choose_alpha_k} follows that for $\eta$ sufficiently small, 
	$$
	16 C^2 k \binom{k}{\lfloor(1-16 C^2 \eta^2 )k\rfloor} \leq 16 C^2 k c_s^{16 C^2 \eta^2 k \ln(\frac{1}{16 C^2 \eta^2}-1)} < c_1^{ \eta^2 k \ln(\frac{1}{16 C^2 \eta^2}-1)} < c_2^{ k \eta^2  \ln\frac{1}{\eta^2}}.
	$$
	Additionally,
	$
	(1-z_4'p)^{-k/4} > c_3^{c_2 p k}
	$.
	Thus, for $\eta$ such that $\eta^2 \ln \frac{1}{\eta} < c_4 p$,  $ 16 C^2 k \binom{k}{\lfloor(1-16 C^2 \eta^2 )k\rfloor} \leq (1-z_4'p)^{-k/4}$ holds.
	Thus, $	\P( \norm{\Omega x}_2  < 2 C  \eta \sqrt{k})
	\le   (1-z_4'p)^{-k/4} \cdot  (1-z_4'p)^{k/2} \le  (1-z_4'p)^{k/4}.
	$
	
\end{proof}

The proof of Theorem \ref{thm:max_on_subspase} is similar to the proof of Theorem \ref{thm:norm}.

\begin{theorem}[Maximum value in a subspace]
    \label{thm:max_on_subspase}
    Let $U \subset \R^n $ be a linear subspace of dimension $r$. Let $\Omega$ be a $k \times n$ random matrix where  $n \geq k > r$ and
     $k = \mathcal{O}(r)$ is sufficiently large. Assume the entries of
     $\Omega$ are i.i.d centered sub-Gaussian random variables.
    Then, for $t \ge C_0$ we have
    $$
    \P \left( \max_{x\in U, \|x\| = 1}\|\Omega x\| > t \sqrt{k} \right) \le e^{-c_0 t^2 k}
   .
    $$
\end{theorem}

\begin{proof}
    Let $\NN$ be a $(1/2)$-net of the $r$-dimensional unit sphere of the image of $U$. 
    Let $\MM $ be a  $(1/2)$-net of the $k$-dimensional unit sphere of the image of $\Omega$. 
    For any $u \in U$ where $\|u\| = 1$, we can choose  $x \in \NN$  such that $\norm{x-u}_2 <1/2$.
    Then,
    $$
    \norm{\Omega u}_2 \le \norm{\Omega x}_2+ \norm{x-u}_2 \max\limits_{u_1 \in U, \|u_1\|=1} \norm{\Omega u_1}
    .
    $$
   Thus,
    $$
    \max\limits_{u_1 \in U, \|u_1\|=1} \norm{\Omega u_1}  \le \norm{\Omega x}_2+\frac{1}{2} \max\limits_{u_1 \in U, \|u_1\|=1} \norm{\Omega
    u_1}.
    $$
    This shows that $\norm{\Omega } \le 2 \sup_{x \in \NN} \norm{\Omega x}_2 = 2 \sup_{x \in \NN} \sup_{v \in S^{k-1}} \pr{\Omega x}{v}$.
    In a similar way, by approximating $v$ with an element from
    $\MM $  we get
    $$
    \sup\limits_{x \in \NN, v \in S^{k-1}} \pr{\Omega x}{v} \leq \sup\limits_{x \in \NN, v \in \MM} \pr{\Omega x}{v} + \frac{1}{2}\sup\limits_{x \in \NN, v' \in S^{k-1}} \pr{\Omega
    x}{v}.
    $$
    We obtain
    $
    \norm{\Omega} \le 4 \max_{ x \in \NN, \ y \in \MM} |\pr{\Omega x}{y}|.
    $
    By Lemma \ref{lem:eps_net_size}, we can choose these
    nets to be $ |\NN|  \le 6^r$ and $|\MM|  \le 6^k$.

    By Theorem \ref{thm:Hoeffding}, for  every $x \in \NN$ and $y \in \MM$,
    the random variable $\pr{\Omega x}{y} =\sum_{j=1}^k \sum_{k=1}^n a_{j,k} y_j x_k $ is sub-Gaussian,
    i.e. for $t>0$
    $$
    \P \big( |\pr{\Omega x}{y} | > t \sqrt{k} \big) \le C_2 e^{-c_1 t^2 k}
    .
    $$
    By taking the union bound we get
    $$
    \begin{array}{ll}
    \P \big( \|\Omega\|_2 > t \sqrt{k} \big)
    &\le  |\NN| |\MM|
    \P \left( |\left\langle \Omega x,y \right\rangle |  > t \sqrt{k}/4, x \in \NN, \,
    y \in \NN  \right) \\
    &\le  6^k \cdot 6^r \cdot  C_2 e^{-c_1/16 t^2 k}
    \le C_2  e^{-c_0 t^2 k},
    \end{array}
    $$
    provided that $t \ge C_0$ for an appropriately chosen constant $C_0>0$.
    This completes the proof.
\end{proof}

By combining Theorem \ref{thm:max_on_subspase} with the $\e$-net
argument and Lemmas \ref{lem:max_val_one_vec_incomp} and \ref{lem:max_val_one_vec} we obtain an estimate for the smallest value of $\|\Omega v\|$ for $v$ in a subspace of dimension $r$.
\begin{theorem}[Smallest value on a subspace]\label{thm:rectangular_small_val_subspace}
	There are constants $M$ and $D$ such that for any $n,r \in \N$, $p \in \R, 0<p<1$,  $n>r$, and for any $r$ dimensional linear subspace $U \subset \R^n$, if $k > D\log\left(\frac{1}{p}\right)\left(r + \frac{1}{p^3}\right)$ then for $\Omega \in M_{k\times n}$ with	centered sub-Gaussian random i.i.d entries as in Definition \ref{def:sub_gauss_X},
	\begin{equation} \label{eq:prob_smallest_value}
	\P \left( \min_{ x \in U, \|x\| = 1} \norm{\Omega x}_2 \le M \eta \sqrt{k}  \right)
	\ll 1
	\end{equation}
	holds for $\eta < \mathcal{O}(\sqrt{p})$.
\end{theorem}
\begin{proof}
	The proof is divided into three steps. In steps 1 and 2  $\Omega$ is bounded on incompressible and compressible vectors, respectively, and in step 3 these results are joined to complete the proof.
	We set $M$ to be $M > \alpha C_0$ where $C_0$ comes from Theorem \ref{thm:max_on_subspase}, and $\alpha$ from Lemma \ref{lem:max_val_one_vec_incomp}, such that $e^{-c_0 \frac{M}{\alpha}^2 k}$ from Theorem  \ref{thm:max_on_subspase} is sufficiently small.
	\begin{description}
			\item[Step 1:]
			Let $\NN $ be a $\alpha \eta$ - net of the set of $(\e_c, \eta)$-incompressible vectors in the image of $U$.
			the number of vectors in $\NN $ is bounded by $\left(\frac{3}{\alpha \eta}\right)^r$. 
			From Lemma \ref{lem:max_val_one_vec_incomp} with $C = \frac{M}{\alpha} $ follows that for any vector $x\in \NN $ and for $\e_c = \e_0 \eta \frac{1}{\sqrt{p}}$,
			$$
			\P ( \norm{\Omega x}_2  < 2M \eta \sqrt{k}) \le ( C_1 \e_0 E(Z^3))^{k/4}.
			$$
			Thus, by the union bound with failure probability of not more than 
			\begin{equation} \label{eq:prob_incompressible_vecs}
			\left(\frac{3}{\alpha \eta}\right)^r \cdot \left( C_1 \e_0 E(Z^3)\right)^{k/4}
			\end{equation}
			the following 
			\begin{equation} \label{eq:min_incompressible_net}
			\min_{x \in \NN} \norm{\Omega x}_2 \ge 2M \alpha \eta \sqrt{k}
			\end{equation}
			holds.
			Since $\NN$ is an $\alpha \eta$-net of the $(\e_c, \eta)$-incompressible vectors in the image of $U$, with the probability given in Eq. \eqref{eq:prob_incompressible_vecs}, then Eq.\eqref{eq:min_incompressible_net} holds. By Theorem \ref{thm:max_on_subspase} we have, for any incompressible vector $y$, 
			$$
			\|\Omega y\| \geq \min_{x \in \NN } \norm{\Omega x}_2 - \alpha \eta \|\Omega\| \geq 2M \alpha \eta \sqrt{k} - M \alpha \eta \sqrt{k}  = M \alpha \eta \sqrt{k}.
			$$
			
			\item[Step 2:]
			Let $\MM $ be a $\eta$ - net of the set  of $(\e_c, \eta)$-compressible vectors in the image of $U$.
			The number of vectors in $\MM $ is bounded by Lemma \ref{lem:eps_net_compressible_size} with  $r^{\frac{1}{\e_c^2}} \eta^{r-\frac{1}{\e_c^2}} \left(\frac{1}{\eta}\right)^r = r^{\frac{1}{\e_c^2}} \eta^{-\frac{1}{\e_c^2}}$. From Lemma \ref{lem:max_val_one_vec} with $C = M $ it follows that for any vector $x \in \MM $, 
			$$
			\P ( \norm{\Omega x}_2  < 2 M \eta \sqrt{k}) \le (1-z_4'p)^{k/4}.
			$$	
			Thus, by the union bound with failure probability of not more than 
			
			\begin{equation} \label{eq:prob_compressible_vecs}
			r^{\frac{1}{\e_c^2}} \eta^{-\frac{1}{\e_c^2}} \cdot (1-z_4'p)^{k/4}, 
			\end{equation}
			the following 
			\begin{equation} \label{eq:min_compressible_net}
			\min_{x \in \MM } \norm{\Omega x}_2 \ge 2M \eta \sqrt{k}
			\end{equation}
			holds.
			Since $\MM$ is an $\eta$-net of the $(\e_c, \eta)$-compressible vectors in the image of $U$, with the probability given in Eq.\eqref{eq:prob_compressible_vecs}, then Eq.  \eqref{eq:min_compressible_net} holds. For any compressible vector $y$ we have
			$$
			\|\Omega y\| \geq \min_{x \in \NN } \norm{\Omega x}_2 - \eta \|\Omega\| \geq 2M \eta \sqrt{k} - M \eta \sqrt{k}  = M \eta \sqrt{k}.
			$$
			
			Thus, if Eqs. \eqref{eq:prob_compressible_vecs} and \eqref{eq:prob_incompressible_vecs} are small enough, then, by Theorem \ref{thm:max_on_subspase}
			$$
			\min\limits_{y \in U, \|y\| = 1} \|\Omega y \| \geq M \alpha \e_c \sqrt{k}.
			$$

			\item[Step 3:]
			The probabilities in Eqs. \eqref{eq:prob_incompressible_vecs} and \eqref{eq:prob_compressible_vecs} are analyzed next. We have
			$$
			\left(\frac{3}{\alpha \eta}\right)^r \cdot \left( C_1 \e_0 E(Z^3)\right)^{k/4} =e^{r \log(\frac{3}{\alpha \eta}) - k/4 \log(\frac{1}{C_1\e_0 z_3})} 
			$$
			
			$$
			r^{\frac{1}{\e_c^2}} \eta^{-\frac{1}{\e_c^2}} \cdot (1-z_4'p)^{k/4} \leq e^{\frac{1}{\e_0^2 \eta^2 p}\log(r) +\frac{1}{\e_0^2 \eta^2 p}\log(\frac{1}{\eta}) -  c_1pk} 
			$$
			for some $c_1$ that depends only on $z_4$. Lemma \ref{lem:max_val_one_vec} holds for $\eta = p^{1/2 - \epsilon}$ and the probabilities in Eqs.   \eqref{eq:prob_incompressible_vecs} and \eqref{eq:prob_compressible_vecs} are less than 
			$$
			e^{r \log(\frac{3}{\alpha \e_0 \eta \sqrt{p}}) - k/4 \log(\frac{1}{M \e_0 z_3})} < 
			e^{r \log(\frac{c_5}{p}) - c_6 k} 
			$$
			and
			$$
			e^{\frac{1}{\e_0^2 \eta^2 p}\log(r) +\frac{1}{\e_0^2 \eta^2 p}\log(\frac{1}{\eta}) -  c_1pk}  < e^{c_7\frac{1}{p^2}\log(r) +c_8\frac{1}{p^2}\log(\frac{1}{\eta}) -  c_1pk} 
			$$
			for constants $c_i$. Thus, for Eq. \ref{eq:prob_smallest_value} to hold, $k$ has to satisfy
			\begin{equation} \label{eq:k_bound1}
			c_9 r \log(\frac{c_5}{p}) \ll k
			\end{equation}
			and
			\begin{equation}\label{eq:k_bound2}
			c_{10}\frac{1}{p^3}\log(r) +c_{11}\frac{1}{p^3}\log(\frac{1}{p}) \ll  k.
			\end{equation}

			Note that $\frac{1}{p^3}\log(r)$ is bounded by $\mathcal{O}(r \log(\frac{c_5}{p}))$ or $\mathcal{O}(\frac{1}{p^3}\log(\frac{1}{p}))$. Thus, there exists some constant $D$ such that Eqs. \ref{eq:k_bound1} and \ref{eq:k_bound2} are equivalent to
			\begin{equation*}
			D\log\left(\frac{1}{p}\right)\left(r + \frac{1}{p^3}\right) \ll  k.
			\end{equation*}
	\end{description}

\end{proof}

\section{Approximated matrix decompositions}
\subsection{Randomized SVD using sparse projections} \label{sec:Sparse_Randomized_SVD}
We present an algorithm that approximates the SVD decomposition of
any matrix $A$. We first recall a known result (e.g. see \cite{halko2011finding}).

\begin{theorem}[Theorem 11.2 in \cite{halko2011finding}]\label{thm:A_QQA}
	Let $A$ be an $m\times n$ matrix with singular values $\sigma_1, \ldots, \sigma_n$ in descending order. For any integer $0<r<m$, let $\Omega$ be a $n \times k $
	random matrix.
	Denote $Y = A\Omega$ and $Y = QR$ where $Q$ is a matrix with orthonormal columns and $R$ is a full rank
	triangular matrix. If for any subspace $U \subset \R^n$ of dimension $k$, $\min\limits_{x \in U} \|\Omega x\|_2$ and  $\|\Omega \|_2$ are bounded from below and from above, respectively, with high probability,  
	Then, with high probability,
	\begin{equation}\label{eq:A_QQA}
	\|A-QQ^*A\|_2 \leq \mathcal{O}_\sigma(\sigma_{r+1})
	\end{equation}
	and 
	\begin{equation}\label{eq:A_QQA_F}
	\|A-QQ^*A\|_F \leq \mathcal{O}_\sigma(\Delta_{r+1}).
	\end{equation}
\end{theorem}
\begin{remark}
	Note that the notation $\mathcal{O}_\sigma(\sigma_{r+1})$ means that the error does not depend on the singular values except of having a linear dependency on $\sigma_{r+1}$. Dependency exists on $n$ and $k$.	
\end{remark}

\begin{rem}\label{rem:omega1_omega2}
Note that if $\Omega_1 \in M_{l \times n}$ and $\Omega_2\in M_{k
\times l}$  satisfy Eq. \eqref{eq:A_QQA} for $A$ of size $m \times n$ and $m \times l$ and $r \in \NN$, respectively, then $\Omega = \Omega_2 \Omega_1$ also satisfies Eq. \eqref{eq:A_QQA} for $A \in M_{m \times n}$ and a rank $r$. This fact is important since it enables us to combine random
matrices by utilizing for example a subsampled randomized
Fourier transform (SRFT) \cite{woolfe2008fast} matrix or a Gaussian matrix with sub-Gaussian matrix. Similar statement is introduced in \cite{clarkson2013low} as Fact 45.
\end{rem}

From Theorem \ref{thm:A_QQA} it follows that the randomized SVD  Algorithm 5.1 in \cite{halko2011finding} is valid for sub-Gaussian matrices.
This algorithm does not take advantage of the fact that $\Omega$ can
be a sparse matrix. Thus, Algorithm 5.1 can be adapted similarly to
the algorithm in Theorem 47 \cite{clarkson2013low} and to the LU
decomposition algorithm \cite{sparseLU}. For SVD
approximation to be of rank $r$, we use the following version of
Weyl's inequality:
\begin{theorem}[Weyl inequality for singular values]
    Let $A,B \in M_{m \times n}$. If $\|A-B\|_2 \leq \e$, then for $1 \leq k \leq \min(m,n)$,  $| \sigma_k(A) - \sigma_k(B)| \leq
    \e$ holds.
\end{theorem}

\begin{proof}
    We prove it by using the min-max principle that for any matrix $A \in M_{m \times n}$, if $\|A-B\|_2 \leq \e$ then $| \sigma_k(A) - \sigma_k(B)| \leq
    \e$.

    The min-max principle states that
    $$
    \sigma_k(A) = \max\limits_{\substack{S\\ \dim S = n-k+1}} \min\limits_{\substack{x\in S,\\ \|x\| = 1}} \|Ax\|.
    $$
    For any matrix $S$ of dimension $n-k+1$, we show that there exists a vector such that $\|Bx\| \leq \sigma_k(A) + \e$.
    For any such $S$, there is a vector $x \in S$ such that  $\|Ax\| = \sigma_k(A)$. Note that
    $$
    \e \geq \|A-B\| \geq \|(A-B)x\| \geq | \|Ax\| - \|Bx\| | = |\sigma_k(A) -  \|B\| |.
    $$
    Thus, for any $S$ of dimension $n-k+1$,  $\min\limits_{x\in S, \|x\| = 1} \|Bx\| \leq \sigma_k(A) + \e$.
    Therefore,
    \begin{equation}\label{eq:sigB_leq_sigA_eps}
    \sigma_k(B) \leq \sigma_k(A) + \e.
    \end{equation}
    By repeating these considerations symmetrically for $A$ with respect to $B$, we have that   $\sigma_k(A) \leq \sigma_k(B) +
    \e$.
    Together with Eq. \ref{eq:sigB_leq_sigA_eps}, we get $| \sigma_k(A) - \sigma_k(B)| \leq
    \e$.
\end{proof}

\begin{corollary}\label{cor:rank_r_approx}
    If  $\| A-U\Sigma V^* \|_2 \le \mathcal{O}_\sigma(\sigma_{r+1}(A))$, then  $\| A-U[\Sigma]_r V^* \|_2 \le \mathcal{O}_\sigma(\sigma_{r+1}(A))$ where $[\Sigma]_r$ is the best rank $r$ approximation of $\Sigma$.
\end{corollary}
\begin{proof}
    \begin{eqnarray}
    \| A - U [\Sigma]_r V^* \|_2 & = &\| A  -  U \Sigma V^*+U\Sigma V^* -  U [\Sigma]_r V^* \|_2 \notag \\
    &\leq &\| A  -  U \Sigma V^* \|_2   +   \| U\Sigma V^* -  U [\Sigma]_r V^* \|_2\notag   \\
    &\leq & O(\sigma_{r+1}(A)) +  \sigma_{r+1}(B)\notag \\
    &\leq & O(\sigma_{r+1}(A)) +  \sigma_{r+1}(A) +  O(\sigma_{r+1}(A)) \notag \\
    & =   & O(\sigma_{r+1}(A)) \notag .
    \end{eqnarray}
\end{proof}

Algorithm \ref{alg:sparse_randomized_SVD} describes a randomized SVD
decomposition for getting a rank $r$ approximation. This approximation
generates the error $\mathcal{O}_\sigma(\sigma_{r+1}(A))$.
Theorem \ref{thm:correctnes_SVD} proves that the algorithm is correct for any matrix distribution that holds the conditions of Theorem \ref{thm:A_QQA}. Its complexity is evaluated in section \ref{ssec:SVD_complexity}.
Numerical results are given in section \ref{subsec:numerical_results}.

\begin{algorithm}[H]
    \caption{Sub-Gaussian-based Randomized SVD Decomposition}
    \label{alg:sparse_randomized_SVD}
    \textbf{Input:} $A$ matrix of size $m \times n$ to decompose, $r$  desired rank, $k_1,k_2,l$ number of columns to use.\\
    \textbf{Output:} Matrices $U, \Sigma, V$ such that $\Vert A-U\Sigma V^* \Vert_2 \le \mathcal{O}_\sigma(\sigma_{r+1}(A))$ where $U$ and $V$ are matrices with orthonormal columns.
    \begin{algorithmic}[1]
        \STATE Create a random sub-Gaussian matrix $\Omega_1$ of size $k_1 \times n$.
        \STATE Create a random Gaussian matrix $\Omega_1^\prime$ of size $l \times k_1$.
        \STATE Compute $B = A\Omega_1^*\Omega_1^{\prime*}$ ($B\in M_{m\times l}$).
        \STATE Compute the QR decomposition: $B = QR$, $Q\in M_{m\times k_1}$ with orthonormal columns, $R\in M_{k_1\times k_1}$ is a full rank upper triangular matrix.
        \STATE Create a random sub-Gaussian matrix $\Omega_2$ of size $k_2 \times m$.
        \STATE Compute $\Omega_2 Q$, $\Omega_2 A$ and $(\Omega_2 Q)^{\dagger}$.
        \STATE Compute the SVD of $(\Omega_2 Q)^{\dagger}\Omega_2 A = \tilde{U}_1\Sigma_1 V_1^*$.
        \STATE $\tilde{U} \gets \tilde{U}_1(:,1:r)$.
        \STATE $\Sigma \gets \Sigma_1(1:r,1:r)$.
        \STATE $V \gets V_1(:,1:r)$.
        \STATE $U \gets Q\tilde{U}$.
    \end{algorithmic}
\end{algorithm}

\begin{theorem}\label{thm:correctnes_SVD}
    Assume that $A$ is a matrix of size $m \times n$ where $m<n$ and $r<m$. Then for $k_1, k_2 = \mathcal{O}\left( \log\left(\frac{1}{p}\right)\left(r + \frac{1}{p^3}\right) \right)$and $l = \mathcal{O}(r)$, Algorithm \ref{alg:sparse_randomized_SVD} outputs $U,\Sigma$ and $V$ such that $ \Vert A-U\Sigma V^* \Vert_2 \le \mathcal{O}_\sigma(\sigma_{r+1}(A))$.
\end{theorem}
\begin{proof}

For a matrix $A \in M_{m \times n}$, let $\Omega_1 \in M_{k_1 \times n}$ be a sub-Gaussian matrix and let $\Omega_1^\prime \in M_{l \times k_1}$ be a random Gaussian matrix. Denote the QR-decomposition of
$A\Omega_1^*\Omega_1^{\prime*} \in M_{m \times l}$ by $QR = A\Omega_1^* \Omega_1^{\prime*}$. From
Theorem \ref{thm:A_QQA}, Remark \ref{rem:omega1_omega2} and Theorem 10.8 of \cite{halko2011finding}, it follows that for $l = \mathcal{O}(r)$
\begin{equation}\label{eq:A_QQA2}
\|QQ^*A - A\|_2 \leq \mathcal{O}_\sigma(\sigma_{r+1}) 
\end{equation}
holds.
From Theorem \ref{thm:rectangular_small_val_subspace} it follows that
for a  sub-Gaussian matrix $\Omega_2 \in M_{k_2 \times m}$, where
$k_2 = \mathcal{O}(k_1)$, the matrix $\Omega_2 Q$ is invertible
from the left, namely $(\Omega_2 Q)^\dagger \Omega_2 Q = I_{k_1
\times k_1}$. Thus, $ \|QQ^*A - A\|_2 = \|Q(\Omega_2
Q)^{\dagger}(\Omega_2 Q)Q^*A - A\|_{2} . $ From the construction $
\|U_1 \Sigma_1 V_1^* - A\|_2 =\|Q \tilde{U}\Sigma V^* - A\|_2
=\|Q(\Omega_2 Q)^{\dagger}\Omega_2A - A\|_2. $  $\|Q(\Omega_2 Q)^{\dagger}\Omega_2A - A\|_2$
is bounded in the following way:

\begin{eqnarray}
    \|Q(\Omega_2 Q)^{\dagger}\Omega_2 A- A\|_2 & \leq & \|Q(\Omega_2 Q)^{\dagger}\Omega_2 A - Q(\Omega_2 Q)^{\dagger}(\Omega_2 Q)Q^*A \notag\\
    &      & + Q(\Omega_2 Q)^{\dagger}(\Omega_2 Q)Q^*A - A\|_2 \notag\\
    & =    & \|Q(\Omega_2 Q)^{\dagger}\Omega_2 (A -QQ^*A) + QQ^*A - A\|_2 \notag \\
    \mbox{by the triangle inequality}
    & \leq & \|Q(\Omega_2 Q)^{\dagger}\Omega_2 (A-QQ^*A)\|_2 +\|QQ^*A - A\|_2 \notag\\
    \mbox{since } \|AB\|_2 \leq \|A\|_2\|B\|_2
    & \leq & \|Q(\Omega_2 Q)^{\dagger}\Omega_2 \|_2\|A-QQ^*A\|_2 +\|QQ^*A - A\|_2\notag \\
    & = & (\|(\Omega_2 Q)^{\dagger}\Omega_2 \|_2 + 1)\|A-QQ^*A\|_2\notag\\
    \mbox{since } \|AB\|_2 \leq \|A\|_2\|B\|_2
    & \leq & (\|(\Omega_2 Q)^{\dagger}\|_2\|\Omega_2 \|_2 + 1)\|A-QQ^*A\|_2
    \label{eq:Q_omega_Q_omega_A__A}.
\end{eqnarray}

From Theorem \ref{thm:rectangular_small_val_subspace} we have $ \|
(\Omega_2  Q )^{\dagger}\|_2 \leq 1/(c_1 \sqrt{k_2})$. From Theorem
\ref{thm:norm} we have $\|\Omega_2 \|_2 \leq C_0\sqrt{n}$. Thus,
from Eq. \eqref{eq:Q_omega_Q_omega_A__A} it follows that
$$
\|Q(\Omega_2 Q)^{\dagger}\Omega_2 A- A\|_2 \leq \left(\frac{C_0}{c_1}
\sqrt{\frac{n}{k_2}} + 1\right)\|A-QQ^*A\|_2 .
$$
Together with Eq. \eqref{eq:A_QQA2} we get that $ \|U_1 \Sigma_1 V_1^* -
A\|_2  \leq \mathcal{O}_\sigma(\sigma_{r+1}) $.

From the result of Corollary \ref{cor:rank_r_approx} we get $ \|U
\Sigma V^* - A\|_2  \leq \mathcal{O}_\sigma(\sigma_{r+1}).
$
\end{proof}
\begin{remark}
A bound for the Frobenius norm $ \Vert A-U\Sigma V^* \Vert_F \le \mathcal{O}_\sigma(\Delta_{r+1}(A))$ is reached similarly by using Eq. \eqref{eq:A_QQA_F}.
\end{remark}
\subsubsection{Computational Complexity of Algorithm \ref{alg:sparse_randomized_SVD}}\label{ssec:SVD_complexity}
For computational complexity estimation and implementation, the internal random matrix distribution of the algorithm is selected as a subclass
of sparse sub-Gaussian matrices. We chose sparse-Gaussian
matrices. Sparse-Gaussian matrices are sparse matrices, where each entry is
i.i.d with probability $1-p $ to be zero and standard Gaussian otherwise. The complexity of each step in Algorithm \ref{alg:sparse_randomized_SVD} is shown in Table \ref{tab:complexity}

\begin{table}[h]
\caption{Complexity of Algorithm \ref{alg:sparse_randomized_SVD} }
\label{tab:complexity}
\centering
\begin{tabular}{|l|c|c|}
    \hline  Step in Algorithm \ref{alg:sparse_randomized_SVD} & $A$ sparse &  $A$ dense \\
    \hline \hline Creation of sparse matrix $\Omega_1$ of size $k_1 \times n$  & $\mathcal{O}(n)$ & $\mathcal{O}(n)$\\
    \hline Computation of $B = A\Omega_1^*$  & $\mathcal{O}(\nnz(A)pk_1 + m k_1 l)$ & $\mathcal{O}(mnpk_1 + m k_1 l)$  \\
    \hline Computation of its QR- decomposition, $B = QR$ & $\mathcal{O}(mk_1^2)$ &$\mathcal{O}(mk_1^2)$\\
    \hline Creation of sparse matrix $\Omega_2$ of size $k_2 \times m$  & $\mathcal{O}(m)$ &$\mathcal{O}(m)$\\
    \hline Computation of $\Omega_2 Q$, $\Omega_2 A$  & $\mathcal{O}(mk_2 + \nnz(A)pk_2)$ & $\mathcal{O}(mk_2 + mnpk_2)$ \\
    \hline Computation of $(\Omega_2 Q)^{\dagger}\Omega_2 A$ & $\mathcal{O}(k_1k_2^2 + nk_2)$ & $\mathcal{O}(k_1k_2^2 + nk_2)$ \\
    \hline Computation of the SVD of $(\Omega_2 Q)^{\dagger}\Omega_2 A$  & $\mathcal{O}(nk_2^2)$ & $\mathcal{O}(nk_2^2)$\\
    \hline
\end{tabular}
\end{table}
\bigskip

The total complexity is $\mathcal{O}(\nnz(A)pk + (m+n)k^2 ) $. For example, for sub-Gaussian random matrices with $p = \mathcal{O}(\frac{1}{\sqrt[3]{r}})$ and $ k = \mathcal{O}(r\log r)$, the complexity is
$ \mathcal{O}(\nnz(A)\sqrt[3]{r^2} \log r + (m+n)(r\log r)^2 )$.

For the OSE defined in \cite{nelson2013osnap}, the asymptotic complexity is the same as in  \cite{nelson2013osnap}. We show in Section \ref{subsec:numerical_results} that although the asymptotic complexity is the same, Algorithm \ref{alg:sparse_randomized_SVD} is faster in practice.

\subsection{Sub-Gaussian based Randomized LU decomposition} \label{sec:Sparse_Randomized_LU}
Theorem \ref{thm:A_QQA} is equivalent to Theorem 3.1 in \cite{sparseLU} where $L_2$ norm is used instead of Frobenius norm. A sub-Gaussian distribution can be used instead of the sparse embedding matrix distribution. Since the correctness proof of the algorithm in \cite{sparseLU} is based on Theorem 3.1, it is also applicable for sub-Gaussian 
matrices. 
\begin{theorem}\label{thm:lu_approx}
    Assume that sub-Gaussian random matrices are used instead of sparse embedding matrices in the approximated rank $r$ LU decomposition in \cite{sparseLU}. Then, for any $r \in \NN$, and for any matrix $A \in M_{m \times n}$, the approximated rank $r$ LU decomposition results in matrices $L$ and $U$ and permutations $P$ and $Q$ such that
    $
    \Vert PAQ- LU \Vert_2 \le \mathcal{O}_\sigma(\sigma_{r+1}(A)).
    $
\end{theorem}
The complexity of the algorithm, as shown in \cite{sparseLU}, is $\mathcal{O}(\nnz(A)pk + (m+n)k^2 )$.

\section{Numerical Results} \label{subsec:numerical_results}
The results in this paper are valid to all types of i.i.d sub-Gaussian
matrices and OSE distributions.  In the current implementation, we used sparse-Gaussian matrices, where each entry in the matrix is i.i.d with probability $p $
to be standard Gaussian and zero otherwise. Note that this distribution is like the distribution in Definition \ref{def:sub_gauss_X} up to a multiplicative constant that does not affect Algorithm \ref{alg:sparse_randomized_SVD}.

We noticed that in practice, for the decomposition of specific matrices, the use of internal random matrix distribution, which has more than one
non-zero entry in a each row (as presented for example here and in
\cite{nelson2013osnap}), results in a much better approximation error than
distributions with one non-zero in each row as in
\cite{clarkson2013low}. This is the reason we use  $p = 3/n$ in the
sparse-Gaussian matrices implementation.

We describe the results from three different experiments. 
All the experiments were implemented on Intel Xeon CPU X5560 2.8GHz. 
All the experiments compare between the running time and the generated error from the following three algorithms in
different scenarios: 
1. The FFT-based algorithm given in \cite{woolfe2008fast}. 
2. The Algorithm from \cite{clarkson2013low}. 
3. Algorithm \ref{alg:sparse_randomized_SVD}.
Although the proven error bounds for Algorithm \ref{alg:sparse_randomized_SVD} are less tight than the bounds for the other algorithms, we see that in practice Algorithm \ref{alg:sparse_randomized_SVD} reaches the same error. 
In all the experiments, the parameters for the different algorithms are chosen such that the reconstruction error rates are similar and aligned to the error from \cite{clarkson2013low} and \cite{woolfe2008fast}. The slowest algorithm has an error that is not smaller than the fastest algorithm.

The experiments that took place are:
\begin{enumerate}
	\item Rank $r$ approximation is computed for a randomly generated full matrix $A\in M_{3000 \times 3000} $
	with singular values that decay exponentially fast from $1$ to  $e^{-50}$.
	Figure \ref{fig:const_n__const_sig} displays the comparison between the running time and the error from rank $r$ approximation 
	from the three algorithms mentioned above.
	The x-axis denotes the rank and the y-axis denotes the running time.
	The results show that for a small rank range \cite{clarkson2013low} is faster than the FFT-based algorithm \cite{woolfe2008fast}. For a larger rank range, the FFT-based algorithm is faster.
	For all ranks, Algorithm  \ref{alg:sparse_randomized_SVD} is  the fastest.

	\begin{figure}[H]
		\makebox[\textwidth][c]{
			\begin{subfigure}[b]{0.6\textwidth}
				\includegraphics[width=\textwidth]{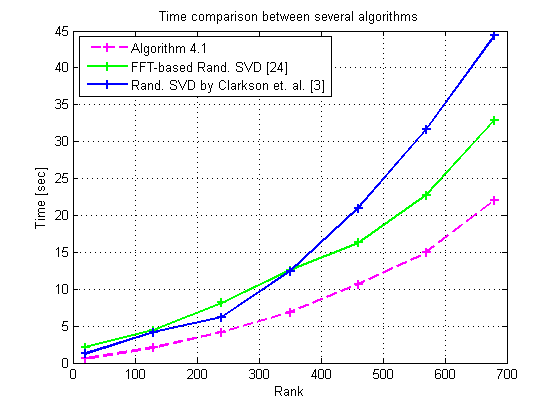}
				\caption{Time}
				\label{fig:const_n__const_sig__n_3000__r_20_110_790__M_1_time}
			\end{subfigure}%
			\begin{subfigure}[b]{0.6\textwidth}
				\includegraphics[width=\textwidth]{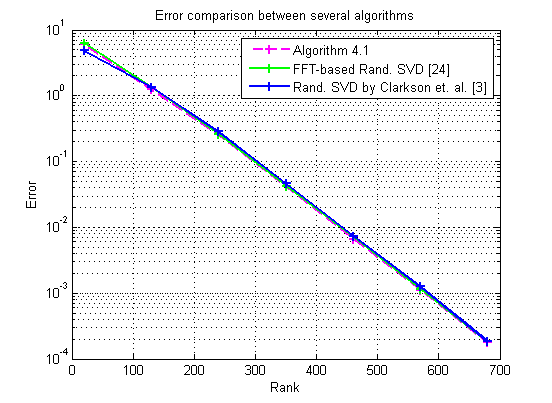}
				\caption{Error}
				\label{fig:const_n__const_sig__n_3000__r_20_110_790__M_1_err}
			\end{subfigure}
		}
		\caption{Results from the  approximation of a  matrix of size $3000 \times 3000$ with exponentially decaying singular values.
			The x-axes in both (a) and (b) denote the rank of the approximation. The y-axis in (a) denotes the run time. The y-axis in (b) denotes the error from the rank approximation.}
		\label{fig:const_n__const_sig}
	\end{figure}

	\item Rank $r$ approximation is computed for a randomly generated full matrix $A\in M_{3000 \times 3000} $ where the first $r$ singular values are 1
	and the other singular values  decay exponentially fast from $e^{-5}$ to  $e^{-50}$.
	Figure \ref{fig:const_n__step_sig__n_3000__r_20_110_790__M_1_time} displays the comparison between the running time for rank $r$ approximation
	for the three algorithms mentioned above. x-axis denotes the rank and the y-axis denotes the running
	time.
	As in experiment 1, for a small rank range, \cite{clarkson2013low} is faster than the FFT-based algorithm \cite{woolfe2008fast}. For a larger rank range, the FFT-based algorithm is
	faster than \cite{clarkson2013low}.
	For all ranks, Algorithm  \ref{alg:sparse_randomized_SVD} is the fastest.

	\begin{figure}[H]
		\makebox[\textwidth][c]{
			\begin{subfigure}[b]{0.6\textwidth}
				\includegraphics[width=\textwidth]{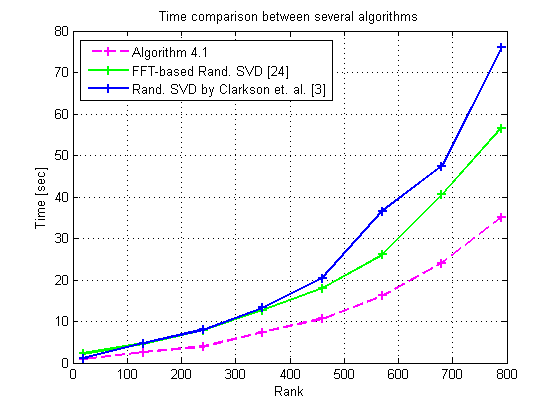}
				\caption{Time}
				\label{fig:const_n__step_sig__n_3000__r_20_110_790__M_1_time}
			\end{subfigure}%
			~
			\begin{subfigure}[b]{0.6\textwidth}
				\includegraphics[width=\textwidth]{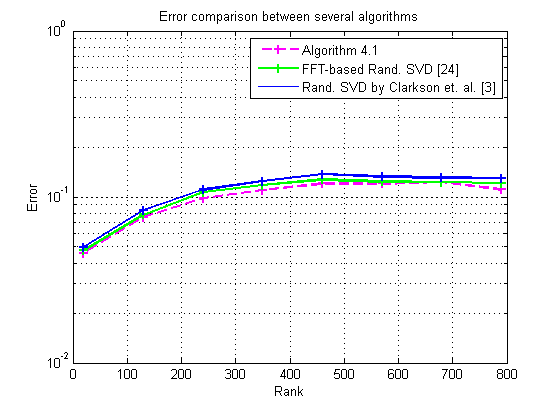}
				\caption{Error}
				\label{fig:const_n__step_sig__n_3000__r_20_110_790__M_1_err}
			\end{subfigure}
		}
		\caption{Results from the approximation of a matrix of size $3000 \times 3000$ with different numerical ranks. The x-axes in both (a) and (b) denote the numerical rank. The y-axis in (a) denotes the run time. The y-axis in (b) denotes the rank approximation error.}
		\label{fig:const_n__step_sig__n_3000__r_30_300}
	\end{figure}

	\item Rank 300 approximation of a randomly generated full matrix $A\in M_{n \times n} $ is computed when the first $300$ singular values are 1
	and the other singular values  decay exponentially fast from $e^{-5}$ to  $e^{-50}$.
	Figure \ref{fig:grow_n__step_sig} displays the comparison between the run time for rank $300$ approximation
	from the three algorithms mentioned above. x-axis denotes the rank and y-axis denotes the running time.
	It is noticeable in this experiment that the sparse SVD from \cite{clarkson2013low} is faster than the FFT-based algorithm \cite{woolfe2008fast} when $n$ increases.
	For rank $300$ and for $n \approx 4500$ the algorithm from \cite{clarkson2013low} is faster than the FFT-based algorithm.
	For ranks larger than $300$, a large $n$ is required for the algorithm from \cite{clarkson2013low} to be faster than the FFT-based algorithm. 
	The Sparse SVD Algorithm \ref{alg:sparse_randomized_SVD} presented in this paper is faster for all $n$.
	
	\begin{figure}[H]
		\makebox[\textwidth][c]{
			\begin{subfigure}[b]{0.6\textwidth}
				\includegraphics[width=\textwidth]{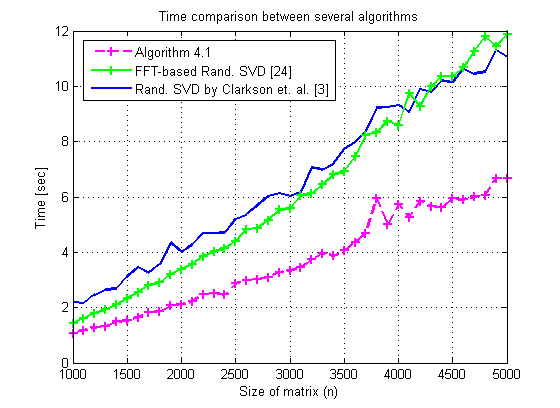}
				\caption{Time}
				\label{fig:grow_n__step_sig__n_1000_100_5000__r_300__M_2_time}
			\end{subfigure}%
			~
			\begin{subfigure}[b]{0.6\textwidth}
				\includegraphics[width=\textwidth]{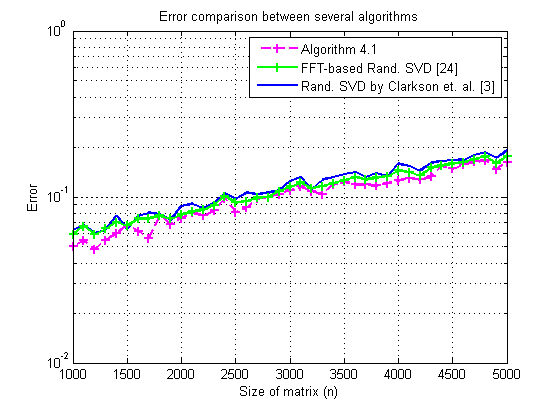}
				\caption{Error}
				\label{fig:grow_n__step_sig__n_1000_100_5000__r_300__M_2_err}
			\end{subfigure}
		}
		\caption{Results from the approximation a matrix of size $n \times n$ , $n = 1000, \ldots, 5000$ with numerical rank 300.
			The x-axis in both (a) and (b) denotes $n$. The y-axis in (a) denotes the run time. The y-axis in (b) denotes the approximation error.}
		\label{fig:grow_n__step_sig}
	\end{figure}
	
\end{enumerate}

In Algorithm \ref{alg:sparse_randomized_SVD}, it is only necessary to apply the matrix $A$ once from the left and once from the right, then $A$ does not have to be stored in memory. Table \ref{tab:large_matrix_results} shows the running time for large matrices that cannot be stored in a computer memory. The matrices we chose have a similar form to the choice in \cite{halko2011algorithm}. We chose $A = F \Sigma F$ where $F$ is the DFT matrix and $\Sigma$ is a diagonal matrix with singular values $\sigma_i$ that decay linearly until $i = 200$ and exponentially from there on. We set $\sum\limits_{i=201}^n \sigma_i$ to be constant in this experiment. Algorithm \ref{alg:sparse_randomized_SVD} is applied to rank $200$ with $k_1 = 500$ and  $k_2 = 700$.

\begin{table}[H]
	\begin{center}
		\begin{tabular}{|c|c|c|c|}
			\hline
			Size ($n$)& Relative Error from Algorithm \ref{alg:sparse_randomized_SVD} & Time for Alg. \ref{alg:sparse_randomized_SVD}(sec) & Time for full SVD \\		\hline
			1,024     & 1.5465 & 1.0011   & 1.5232  \\ \hline
			2,048     & 1.5645 & 1.6236   & 11.3702 \\ \hline
			4,096     & 1.5422 & 2.7653   & 94.6345 \\ \hline
			8,192     & 1.5571 & 5.2999   & 578.2982 \\ \hline
			16,384    & 1.4846 & 12.1065  & 4324.683 \\ \hline
			32,768    & 1.5686 & 26.4022  & \multicolumn{1}{r}{} \\ \cline{1-3}
			65,536    & 1.5074 & 50.6191  & \multicolumn{1}{r}{} \\ \cline{1-3}
			131,072   & 1.4838 & 109.8185 & \multicolumn{1}{r}{} \\ \cline{1-3}
			262,144   & 1.5357 & 205.0068 & \multicolumn{1}{r}{} \\ \cline{1-3}
			524,288   & 1.4854 & 418.4137 & \multicolumn{1}{r}{} \\ \cline{1-3}
			1,048,576 & 1.5240 & 847.8211 & \multicolumn{1}{r}{} \\ \cline{1-3}
		\end{tabular}
		\caption{Comparing running time of Algorithm \ref{alg:sparse_randomized_SVD} to the standard SVD for large matrices. The Relative Error is the ratio between the error from the rank $r$ decomposition and from the $r+1$ singular value.}\label{tab:large_matrix_results} 
	\end{center}
\end{table}
\section*{Conclusion}
\label{sec:conclusion}
We showed that matrices with i.i.d sub-Gaussian entries conserve subspaces and showed the connection between the distribution of the entries and the required size of the matrix. A new algorithm is presented, which yields with high probability, a rank $r$ SVD approximation for an $m \times n$ matrix that achieves an asymptotic complexity of $\mathcal{O}(\nnz(A)pk + (m+n)k^2 )$. Additionally, we showed that the approximated LU algorithm in \cite{sparseLU}, which uses sub-Gaussian random matrices, has a computational complexity of $\mathcal{O}(\nnz(A)pk + (m+n)k^2 )$. We showed in the experiments that although the derived error bounds are not as tight as the bounds from the algorithms in \cite{clarkson2013low,halko2011finding}, in practice, the algorithm in this paper reaches the same error in less time. 

Future work includes non-asymptotic estimation of the algorithm parameters including error estimation improvement to get tighter bounds.

\section*{Acknowledgment}
\noindent This research was partially supported by the Israeli Ministry of Science \& Technology (Grants No. 3-9096, 3-10898), US-Israel Binational Science Foundation (BSF 2012282), Blavatnik Computer Science Research Fund, Blavatink ICRC Funds and  by a Fellowship from Jyv\"{a}skyl\"{a} University. We thank Prof. Jelani Nelson and Dr. Haim Avron for their continues support and constructive remarks.

\bibliographystyle{siam}
\bibliography{SubGaussianDecomp}

\begin{thebibliography}{10}

\bibitem{sparseLU}
{\sc Y.~Aizenbud, G.~Shabat, and A.~Averbuch}, {\em Randomized {LU}
  decomposition using sparse projections}, Preprint,  (2015).

\bibitem{berry1941accuracy}
{\sc A.~C. Berry}, {\em The accuracy of the gaussian approximation to the sum
  of independent variates}, Transactions of the american mathematical society,
  49 (1941), pp.~122--136.

\bibitem{clarkson2013low}
{\sc K.~L. Clarkson and D.~P. Woodruff}, {\em Low rank approximation and
  regression in input sparsity time}, in Proceedings of the 45th annual ACM
  symposium on Symposium on theory of computing, ACM, 2013, pp.~81--90.

\bibitem{cohen2016nearly}
{\sc M.~B. Cohen}, {\em Nearly tight oblivious subspace embeddings by trace
  inequalities}, in Proceedings of the Twenty-Seventh Annual ACM-SIAM Symposium
  on Discrete Algorithms, 2016, pp.~278--287.

\bibitem{dasgupta2010sparse}
{\sc A.~Dasgupta, R.~Kumar, and T.~Sarl{\'o}s}, {\em A sparse johnson:
  Lindenstrauss transform}, in Proceedings of the forty-second ACM symposium on
  Theory of computing, ACM, 2010, pp.~341--350.

\bibitem{dirksen2014dimensionality}
{\sc S.~Dirksen}, {\em Dimensionality reduction with subgaussian matrices: a
  unified theory}, arXiv preprint arXiv:1402.3973,  (2014).

\bibitem{esseen1942liapounoff}
{\sc C.-G. Esseen}, {\em On the Liapounoff limit of error in the theory of
  probability}, Almqvist \& Wiksell, 1942.

\bibitem{halko2011algorithm}
{\sc N.~Halko, P.-G. Martinsson, Y.~Shkolnisky, and M.~Tygert}, {\em An
  algorithm for the principal component analysis of large data sets}, SIAM
  Journal on Scientific computing, 33 (2011), pp.~2580--2594.

\bibitem{halko2011finding}
{\sc N.~Halko, P.-G. Martinsson, and J.~A. Tropp}, {\em Finding structure with
  randomness: Probabilistic algorithms for constructing approximate matrix
  decompositions}, SIAM review, 53 (2011), pp.~217--288.

\bibitem{hoeffding1963probability}
{\sc W.~Hoeffding}, {\em Probability inequalities for sums of bounded random
  variables}, Journal of the American statistical association, 58 (1963),
  pp.~13--30.

\bibitem{johnson1984extensions}
{\sc W.~B. Johnson and J.~Lindenstrauss}, {\em Extensions of lipschitz mappings
  into a hilbert space}, Contemporary mathematics, 26 (1984), p.~1.

\bibitem{kane2014sparser}
{\sc D.~M. Kane and J.~Nelson}, {\em Sparser johnson-lindenstrauss transforms},
  Journal of the ACM (JACM), 61 (2014), p.~4.

\bibitem{litvak2005smallest}
{\sc A.~E. Litvak, A.~Pajor, M.~Rudelson, and N.~Tomczak-Jaegermann}, {\em
  Smallest singular value of random matrices and geometry of random polytopes},
  Advances in Mathematics, 195 (2005), pp.~491--523.

\bibitem{martinsson2011randomized}
{\sc P.-G. Martinsson, V.~Rokhlin, and M.~Tygert}, {\em A randomized algorithm
  for the decomposition of matrices}, Applied and Computational Harmonic
  Analysis, 30 (2011), pp.~47--68.

\bibitem{nelson2013osnap}
{\sc J.~Nelson and H.~L. Nguy{\^e}n}, {\em Osnap: Faster numerical linear
  algebra algorithms via sparser subspace embeddings}, in Foundations of
  Computer Science (FOCS), 2013 IEEE 54th Annual Symposium on, IEEE, 2013,
  pp.~117--126.

\bibitem{nelson2013sparsity}
\leavevmode\vrule height 2pt depth -1.6pt width 23pt, {\em Sparsity lower
  bounds for dimensionality reducing maps}, in Proceedings of the forty-fifth
  annual ACM symposium on Theory of computing, ACM, 2013, pp.~101--110.

\bibitem{rudelson2014recent}
{\sc M.~Rudelson}, {\em Recent developments in non-asymptotic theory of random
  matrices}, Modern Aspects of Random Matrix Theory, 72 (2014), p.~83.

\bibitem{rudelson2008littlewood}
{\sc M.~Rudelson and R.~Vershynin}, {\em The littlewood--offord problem and
  invertibility of random matrices}, Advances in Mathematics, 218 (2008),
  pp.~600--633.

\bibitem{rudelson2009smallest}
\leavevmode\vrule height 2pt depth -1.6pt width 23pt, {\em Smallest singular
  value of a random rectangular matrix}, Communications on Pure and Applied
  Mathematics, 62 (2009), pp.~1707--1739.

\bibitem{shabat2013randomized}
{\sc G.~Shabat, Y.~Shmueli, Y.~Aizenbud, and A.~Averbuch}, {\em Randomized {LU}
  decomposition}, arXiv preprint arXiv:1310.7202,  (2013).

\bibitem{tao2012topics}
{\sc T.~Tao}, {\em Topics in random matrix theory}, vol.~132, American
  Mathematical Soc., 2012.

\bibitem{tropp2011improved}
{\sc J.~A. Tropp}, {\em Improved analysis of the subsampled randomized hadamard
  transform}, Advances in Adaptive Data Analysis, 3 (2011), pp.~115--126.

\bibitem{vershynin2010introduction}
{\sc R.~Vershynin}, {\em Introduction to the non-asymptotic analysis of random
  matrices}, arXiv preprint arXiv:1011.3027,  (2010).

\bibitem{woolfe2008fast}
{\sc F.~Woolfe, E.~Liberty, V.~Rokhlin, and M.~Tygert}, {\em A fast randomized
  algorithm for the approximation of matrices}, Applied and Computational
  Harmonic Analysis, 25 (2008), pp.~335--366.

\end{thebibliography}
\end{document}